\documentclass{article}
\usepackage[utf8]{inputenc}
\usepackage{amsmath, amsthm, amssymb, amsfonts}
\usepackage{bm}
\usepackage{comment}
\usepackage{dsfont}
\usepackage[utf8]{inputenc}
\usepackage{rotate}
\usepackage{tikz}
\usepackage{tikz-cd}
\usepackage[arrow,matrix,curve]{xy} 	
\usepackage{xcolor}
\usepackage{todonotes}
\usepackage{url}

\theoremstyle{plain}
\newtheorem{theorem}{Theorem}

\newtheorem{lemma}[theorem]{Lemma}
\newtheorem{corollary}[theorem]{Corollary}

\theoremstyle{definition}

\newtheorem{remark}[theorem]{Remark}

\newcommand{\NN}{{\mathbb{N}}}

\title{Hammersley Point Sets and Inverse of Star-Discrepancy}
\author{Christian Wei\ss{}}
\date{\today}

\begin{document}

\maketitle

\begin{abstract} We establish the existence of $N$-point sets in dimension $d$ whose star-discrepancy is bounded above by $2.4631832 \sqrt{\frac{d}{N}}$, where the numerical constant improves upon all previously known bounds. This improvement is obtained by combining a recent result by Gnewuch on bracketing numbers in high dimensions with discrepancy bounds for Hammersley point sets due to Atanassov in dimensions $1 \leq d \leq 4$.
\end{abstract}

\section{Introduction}

The construction of sets with small (star-)discrepancy is of considerable interest from both a theoretical and practical standpoint, particularly in the context of high-dimensional numerical integration. Recall that the star-discrepancy of a point set $P \subset [0,1]^d$ with $N \in \mathbb{N}$ elements is for the $d$-dimensional Lebesgue measure $\lambda_d(\cdot)$  defined by
$$D_N^*(P) := \sup_{B \subset [0,1)^d} \left| \frac{\# \left\{ 1 \leq n \leq N \, : \, x_n \in B \right\}}{N} - \lambda_d(B) \right|,$$
where the supremum is taken over all $d$-dimensional intervals $B = [0,b_1) \times \ldots \times [0,b_d)$ with $0 \leq b_i \leq 1$ for $i=1,\ldots,d$. It is widely conjectured that $D_N^*(P) = O(\log(N)^{d-1}/N)$ is the best achievable order for the star-discrepancy of point sets in dimension $d$. However, this conjecture has only been proven in the cases $d\leq2$, see \cite{Sch72}. Point sets that exhibit the presumed optimal behavior of the star-discrepancy are called low-discrepancy point sets.\\[12pt]
As the bound $\log(N)^{d-1}/N$ for low-discrepancy point sets exponentially depends on the dimension, for large $d$ the enumerator is large in comparison to the denominator if $N$ is small. Therefore, the bound is hardly of any use in this setting. It thus makes sense to ask, what is the smallest star-discrepancy achievable for a given $N \in \mathbb{N}$ and to set
$$D^*(N,d) = \inf \{ D_N^*(P) \, : \, P \subset [0,1]^d, \# P = N \},$$
which is also called the $N$-th minimal star-discrepancy in dimension $d$. Furthermore, define the inverse star-discrepancy by
$$N^*(\varepsilon,d):= \inf \{ N \in \mathbb{N} \, : \, D^*(N,d) \leq \varepsilon\},$$
which is the minimum number of points that guarantees a star-discrepancy bound of at most $\varepsilon > 0$. In the seminal paper \cite{HNWW00}, a bound for the smallest achievable star-discrepancy of the form
\begin{align} \label{ineq:disc:sqrt}
D^*(N,d) \leq c \sqrt{\frac{d}{N}},
\end{align}
for some constant $c >0$ was shown without giving an explicit value for $c$. Currently, the best known value for the constant is $c = 2.4968$ according to \cite{GPW21}. In this note, we make the following improvement.
\begin{theorem}\label{thm:improved:c2} For any $d,N \in \mathbb{N}$ we have
$$D^*(N,d) \leq 2.4631832 \sqrt{\frac{d}{N}}$$
and hence
\[
N^*(\varepsilon,d) \le 6.0672715\,d\,\varepsilon^{-2}.
\]
\end{theorem}
The general structure of the proof for Theorem~\ref{thm:improved:c2} is in parallel to \cite{GPW21}. However, we add some additional ideas to achieve our improvement: First, we were able to slightly sharpen the arguments to derive the claim. Second, we include a very recent result on bracketing numbers from \cite{Gnewuch2024}. Last but not least, we employ Hammersley point sets along with bounds on their star-discrepancy from \cite{Ata04} to address the case $d \leq 4$ whereas previously only the case $d=1$ had been treated separately. It turns out that the star-discrepancy of finitely many Hammersley point sets must be computed individually for each dimension $d$ to complete the proof. Without utilizing the result from \cite{Ata04} the number of necessary computations in dimension 4 would be prohibitively large. Even with this result, additional arguments are needed to further reduce the computational effort, making it feasible to carry out the calculations on a standard computer. This double reduction in computational complexity constitutes the main new contribution of the present article. We will discuss the details in the next section.\\[12pt]
Due to the Koksma--Hlawka inequality, see, e.g., \cite{KN74}, the worst-case error when approximating the integral of a function $f \colon [0,1]^d \to \mathbb{R}$ by the average of its evaluations at a point set $\{x_1, \dots, x_N\}$ is bounded by the product of the variation of $f$ in the sense of Hardy and Krause and the star-discrepancy of the point set. Therefore, even small reductions of $N^*(\varepsilon,d)$ are of practical importance if an expensive function $f$, i.e., one that requires considerable computational effort, is evaluated at these points to (numerically) calculate an integral. Alternatively, decreasing the size of $N$ can be motivated by the numerical stability of certain regression problems which depend linearly on the star-discrepancy, see \cite{WN19}. Therefore, precise theoretical bounds and numerical estimates of $N^*(\varepsilon,d)$ are of great interest.
\paragraph{Acknowledgments.} The author would like to thank Francois Cl\'ement for discussions on the content of the paper and for providing the link to the C-code used to perform the calculations of the DEM algorithm. Moreover, he is grateful to Michael Gnewuch for his valuable comments on a preliminary version of this paper. The author also acknowledges the two referees for their careful reading and constructive suggestions, which led to significant improvements in the presentation.

\section{The main result} \label{sec:theory}

\paragraph{Halton Sequences.} As mentioned in the introduction, we will use Hammersley point sets to ensure that the bound $D_N^*(P) \leq 2.4631832 \sqrt{\frac{d}{N}}$ is satisfied for small dimensions $d$. They build upon Halton sequences which we will therefore introduce first. Recall that for an integer $b \geq 2$ the $b$-ary representation of $n \in \NN$ is $n = \sum_{j=0}^\infty e_j b^j$ with $0 \leq e_j = e_j(n) < b$. Based on the radical-inverse function, the van der Corput sequence in base $b$ is defined by $\varphi_b(n)=\sum_{j=0}^\infty e_j b^{-j-1}$ for all $n \in \NN$.
The Halton sequence $(H^b_d)$ in the bases $b=(b_1,\ldots,b_d)$ is then given by $x_n := (\varphi_{b_1}(n),\ldots,\varphi_{b_d}(n))$ for all $n \geq 1$, where $b_1,\ldots,b_d$ are pairwise relatively prime integers. It is well-known that Halton sequences satisfy for all $N \in \mathbb{N}$ and corresponding point sets $P=\{ x_1,\ldots,x_N \}$ the inequality $D_N^*(P) \leq C \log(N)^d/N$  with some constant $C > 0$, which is independent of $N$, see e.g. \cite[Theorem~3.6]{Nie92}.
\paragraph{Lifting of sequences.} Let $(x_n)$ be a sequence in $[0,1]^{d-1}$, and let $D_N^*(x_n)$ denote the $(d-1)$-dimensional star-discrepancy of the set consisting of its first $N$ elements. Then it is possible to construct a point \textit{set} in $[0,1)^d$ with almost the same $d$-dimensional star-discrepancy. This technique originally goes back to \cite{Rot54}. By defining $y_n=((n-1)/N,x_n) \in [0,1)^d$ for $n=1,\ldots,N$, and $P=\{y_1,\ldots,y_{N}\}$ we obtain
\begin{align*} D_N^*(P) \leq \max_{1 \leq M \leq N} \frac{M}{N} D_M^*(x_n) + \frac{1}{N},
\end{align*}
see \cite{Nie92}, Lemma 3.7. For the Halton sequence in bases $b$, this specific point set is known as Hammersley point set and denoted by $\textrm{Ham}_b(N)$. The bound then becomes
\begin{align} \label{ineq:hammer}
     D_N^*(P) \leq c_{d-1} \frac{\log(N)^{d-1}}{N} + \frac{1}{N}.
\end{align}
\paragraph{Bracketing numbers.} Furthermore, our proof relies on the concept of bracketing numbers which we introduce next. It serves as an important tool for finding bounds on the star-discrepancy, compare e.g. \cite{GPW21} and \cite{doerr2005bounds}. Here we follow the definition given in \cite{Gne08}: Let $A \subset [0,1]^d$ and $\delta \in (0,1]$. A finite set of points $\Gamma \subset [0,1]^d$ is called a $\delta$-cover of $A$, if for every $y \in A$, there exist $x,z \in \Gamma \cup \left\{ 0 \right\}$ such that $x \leq y \leq z$ (to be read component-wise) and $\lambda_d([0,z]) - \lambda_d([0,x]) \leq \delta$. For $x \leq y$ the interval $[x,y]$ is defined by $[x,y] := [x_1,y_1] \times [x_2,y_2] \times \ldots \times [x_d,y_d]$. The bracketing number $N_{[\ ]}(A,\delta)$ of $A$ is the smallest number of closed axis-parallel boxes (or brackets) of the form $[x,y]$ with $x,y \in [0,1]^d$, satisfying $\lambda_d([0,y]) - \lambda_d([0,x]) \le \delta$, whose union contains $A$. It is easily verified that the cardinality of a minimal $\delta$-cover of $A$, i.e., with the smallest possible number of sets, is bounded from above by $2N_{[\ ]}(A,\delta)$, compare \cite{Gne08} and \cite{GPW21}. The following result from \cite{Gnewuch2024} improved upon the previously best-known bounds for $d$-dimensional bracketing numbers.
\begin{lemma}[Gnewuch, \cite{Gnewuch2024}, Theorem 2.9] \label{thm:mg} The cardinality of the optimal $\delta$-bracketing cover can be estimated as
$$N_{[ \ ]}\left([0,1]^d,\varepsilon\right) \leq \frac{d^d}{d!}\left( \frac{1}{\delta} \right)^d $$
for $d \geq 3$.
\end{lemma}
\paragraph{Bounds on the star-discrepancy.} If $N$ is large in comparison to $d$, then theoretical star-discrepancy bounds for the Hammersley point set as a low-discrepancy point set yield a better bound than the one in Theorem~\ref{thm:improved:c2} anyhow. For any $d \in \mathbb{N}$, there remain finitely many $n=1,\ldots,N=N(d)$, where the known bounds on the star-discrepancy do not suffice to bound the star-discrepancy by $2.4631832 \sqrt{\frac{d}{n}}$. In these finitely many exceptions it is however possible to find sets which still have the desired property. Nonetheless, $N =N(d)$ is still too large to just naively calculate all necessary star-discrepancies, because the task would not be feasible on a usual computer due to the complexity of the calculation. Therefore, we will explain and use a theoretically justifiable trick which helps to speed up the necessary computer calculations, see Remark~\ref{rem:simplifcation:calculation}.\\[12pt]
The main step towards Theorem~\ref{thm:improved:c2} is the following Lemma~\ref{thm:bound:d5}, which is in large parts analogous to Theorem 3.5 in \cite{GPW21}. In that article, bounds on bracketing numbers were derived using a generalized Faulhaber inequality, and subsequently applied to obtain the mentioned bound on the star-discrepancy. The main difference here is the restriction to the case $d \geq 5$. The remaining cases $d \leq 4$ in the proof of Theorem~\ref{thm:improved:c2} can then be covered by Hammersley point sets and some additional ideas to reduce the required number of discrepancy calculations to be performed on a computer.
\begin{lemma} \label{thm:bound:d5} Let $d,N \in \mathbb{N}$ with $d \geq 5$. Let $X = (X_n)$ be a sequence of uniformly distributed, independent random variables on $[0,1]^d$. Then for every $c\geq 2.4631832$
$$D_N^*(X) \leq c \sqrt{\frac{d}{N}}$$
holds with probability at least $1 - e^{-(1.6728349c^2-10.1495427)\cdot d}$ implying that for every $q \in (0,1)$
$$D_N^*(X) \leq 0.7731673 \sqrt{10.1495427+ \frac{\log\left((1-q)^{-1}\right)}{d}} \cdot \sqrt{\frac{d}{N}}$$
holds with probability at least $q$.
\end{lemma}
\begin{proof}
    Let $\mu \in \mathbb{N}, \mu \geq 2$ be arbitrary and choose a $2^{-\mu}$-cover $\Gamma_\mu$ of minimum size. Recall that its cardinality is bounded by $2N_{[\ ]}(A,\delta)$. Applying Lemma~\ref{thm:mg} and the Stirling formula, see e.g. \cite[II.9]{Fel91}, thus implies
    $$|\Gamma_\mu| \leq 2\frac{1}{\sqrt{2 \pi d}} e^d 2^{\mu d}.$$
    As we want to avoid pure repetition, we will mention and use an intermediate result in the proof of Theorem 3.5 in \cite{GPW21}. To do so, we need to introduce three auxiliary variables. First we set
    $$c_{\mu} := \frac{1}{1-\sqrt{\frac{\mu + 1}{2\mu}}}$$
    for $\mu \in \mathbb{N}, \mu \geq 2$ and for $\tau \geq 0$ define
    $$c_1 := \sqrt{4 \tau \left( 1 + \frac{1}{3c_{\mu}} \right)}.$$
    Now we use the following lemma, which is essentially a combination of Lemma 3.4 and the beginning of the proof of Theorem 3.5 in \cite{GPW21}.
\begin{lemma}[\cite{GPW21}] \label{lem:GPW} Let $X=(X_n)$ be a sequence of uniformly distributed and independent random variables in $[0,1]^d$. Then for $c_0 \geq \sqrt{(\mu-\sigma)/2}$, the inequality 
$$D_N^*(X) \leq c_0 \left(1 + c_1 c_{\mu} \sqrt{\frac{\mu}{2^\mu}} \right) \sqrt{\frac{d}{N}}$$
is satisfied with probability at least
\begin{align} \label{ineq:lemma} 1 - \sqrt{\frac{2}{\pi d}} e^{-(2c_0^2-\mu+\sigma)d}\left( 1 + \frac{e^{-((\mu-\sigma)(\mu \tau -1)+(1-\log(2))\mu-\zeta-\sigma)d}}{1-e^{(-((\mu-\sigma)\tau -\log(2)))d}}\right),
\end{align}
where
$$0 \leq  \sigma \le \mu - d^{-1} \log\left(\frac{2|\Gamma_\mu|}{\sqrt{\frac{2}{\pi d}}}\right)$$
and 
$$\zeta = 1 + \log(2) + \frac{\log(2)}{d}.$$
\end{lemma}
In our situation the two quantities from the lemma are $0 \leq \sigma = \sigma(d) \leq \mu-\log(2^\mu)-1-\frac{\log(2)}{d}$ and $\zeta = 1 + \log(2) + \frac{\log(2)}{d}$. We now want to make sure that the expression in the bracket of \eqref{ineq:lemma} is $\leq \sqrt{\frac{\pi d}{2}}$, so that the probability becomes positive. This is true for $d=5$, $\mu=13$ and $\tau = 0.089494518961035$. We fix the values of $\mu$ and $\tau$ for the remainder of the proof. Moreover, we set $\sigma =  \sigma(5) = \mu-\log(2^\mu)-1-\frac{\log(2)}{5}$ for all dimensions which is possible according to Lemma~\ref{lem:GPW} since $d \geq 5$ and $\sigma(d)$ takes its smallest value for $d=5$. Note that the function in the bracket of \eqref{ineq:lemma} is monotonically decreasing in $d$, because the numerator of the fraction is monotonically decreasing in $d$ while the denominator is monotonically increasing in $d$. Hence, the bound $\leq \sqrt{\frac{\pi d}{2}}$ holds for all $d \geq 5$ and the lower bound on the probability becomes $1-e^{-(2c_0^2-\mu+\sigma)\cdot d}$. Realizing that $\mu-\sigma \approx 10.1495427$ and solving the relation
\[
c = c_0\left( 1+c_1 c_{\mu} \sqrt{\frac{\mu}{2^\mu}}\right)
\]
for $c_0$, we derive the claimed probability. Putting $q:=1-e^{-(1.6728349c^2-10.1495427)\cdot d}$ and solving for $c$ yields the second estimate for the star-discrepancy.
\end{proof}
This allows us to complete the proof of Theorem~\ref{thm:improved:c2}, in which we moreover use the following result of Atanassov in \cite{Ata04}, which we present here for Hammersley point sets, see also \cite[Theorem 3.46]{DP10}. 
\begin{lemma}[Atanassov, \cite{Ata04}] \label{thm:attanasov} Let $\mathrm{Ham}_b(N)$ be the $N$-element Hammersley point set in the bases $b = (b_1,\ldots,b_d)$. Then $D_N^*(\mathrm{Ham}_b(N))$ is bounded from above by
$$\left( \frac{1}{d!} \prod_{j=1}^d \left( \frac{\lfloor b_j/2 \rfloor \log(N)}{\log(b_j)} + d\right) + \sum_{k=0}^{d-1} \frac{b_{k+1}}{k!} \prod_{j=1}^k \left( \frac{\lfloor b_j/2 \rfloor \log(N)}{\log(b_j)}+k\right) +1 \right) \frac{1}{N}.$$
\end{lemma}
\begin{proof}[Proof of Theorem~\ref{thm:improved:c2}]
For $d=1$, the set $P := \{ 1/(2N), 3/(2N), \ldots , (2N-1)/(2N) \}$ satisfies $D_N^*(P) = \frac{1}{2N}$ which is stronger than the claim. In the case $d=2$, the Hammersley point set $\textrm{Ham}_2(N)$ in the base $2$ with $N$ elements satisfies
\[
D_N^*(\textrm{Ham}_2(N)) \leq D_N(\textrm{Ham}_2(N)) \leq \frac{7}{2N} + \frac{1}{2\log(2)} \frac{\log(N)}{N},
\]
according to the simple bound \cite[Theorem~3.8]{Nie92}, which is smaller than $2.463 \sqrt{\frac{2}{N}}$ for $N>1$. In dimension $d=3$, the Hammersley point set $\textrm{Ham}_{2,3}(N)$ in the bases $2, 3$ satisfies
$$D_N^*(\textrm{Ham}_{2,3}(N)) \leq \frac{3}{N} + \frac{1}{N} \left( \frac{1}{2\log(2)} \log(N) + \frac{3}{2}\right)\left( \frac{1}{\log(3)} \log(N) + 2\right)$$
which is smaller that $2.463 \sqrt{\frac{3}{N}}$ for $N > 28$. However, for $N \leq 28$, the actual star-discrepancy of the Hammersley point set can be calculated with the help of a computer (or also by hand if desired) to see that the claimed inequality is actually true for all $N \in \mathbb{N}$.\\[12pt]
For the case $d=4$ the mentioned theoretical bounds on the star-discrepancy of the Hammersley point set from \cite[Theorem~3.8]{Nie92} does not suffice because it would not be feasible to calculate the exact star-discrepancies for the missing $N$, even with a very fast computer. \\[12pt]
However, applying Lemma~\ref{thm:attanasov} implies that the Hammersley point set in the bases $2, 3, 5$ satisfies $D_N^*(\textrm{Ham}_{2,3,5}(N)) \leq 2.463\sqrt{\frac{4}{N}}$ for $N > 23,934$. The finitely many (relatively few) excluded point sets can now be checked on a computer to satisfy the desired inequality, see Remark~\ref{rem:simplifcation:calculation}. This completes the proof.
\end{proof}
Theoretically, our method could be extended to higher dimensions. For the case $d=5$ the  theoretical bound on the star-discrepancy of the Hammersley point set from Lemma~\ref{thm:attanasov} however guarantees $D_N^*(\textrm{Ham}_{2,3,5,7}(N)) \leq 2.463\sqrt{\frac{5}{N}}$ only for $N > 14,416,626$. Again, it would be necessary to check the remaining exceptions on a computer. As the runtime of the so-called DEM algorithm, which is the best exact algorithm to calculate the star-discrepancy, is of order $O(N^{d/2+1})$, see \cite{DEM96}, our approach hence seems to be infeasible from dimension $5$ on: as the dimension increases, the number of exceptions grows, and their star-discrepancies become progressively harder to compute. If we nonetheless assume that the approach could be conducted up to e.g. $d_0 = 9$, then the constant $c$ in \eqref{ineq:disc:sqrt} would go down to approximately $2.4543$.
\begin{remark} \label{rem:simplifcation:calculation} The following triangle inequality for the star-discrepancy accelerates the computer computation by eliminating the need to calculate star-discrepancies for many values of $N$: Suppose that we have a sequence $(x_n)$ and that $D_{N_0}^*(x_n)$ is known and that we want to calculate $D_{N_1}^*(x_n)$ for $N_1 = \alpha N_0$ with $\alpha > 1$ such that $N_1 \in \mathbb{N}$. Then according to \cite{KN74}, Theorem 2.6 on p. 115, it holds that
$$D_{N_1}^*(x_n) \leq \frac{1}{\alpha} D_{N_0}^*(x_n) + \frac{\alpha-1}{\alpha}.$$
Thus, the star-discrepancy can increase by at most $\frac{\alpha-1}{\alpha}$. If $b = 2.463\sqrt{\frac{d}{N}} - D_{N_0}^*(x_n)$, then we may choose $\alpha = 1/(1-b)$. For instance for $N_0=5,000$ and $d=4$ we have $D_{N_0}^*(\textrm{Ham}_{2,3,5}(N)) = 0.0045$ for the (four-dimensional) Hammersley point set in the bases $b=(2, 3, 5)$ and $2.463\sqrt{\frac{d}{N_0}} = 0.06966$. Hence we could directly jump to $N_1 = 5,348$ elements and calculate their star-discrepancy if we dealt with a sequence. The gain is that it is not necessary to calculate any discrepancy in the range $5,001 \leq  N \leq 5,347$.

For the Hammersley point set, we need however to account for the fact that it is not a sequence: we know that the star-discrepancy of the underlying Halton sequence in dimension $3$ is at most $D_{N_0}^*(\mathrm{Ham}_{2,3,5}(N))$ as well and the presented argument works for the Halton sequence when moving from $N_0$ to $N_1$. For the Hammersley point set with $N_1$ points, the star-discrepancy might be bigger by $1/N_1 < 1/N_0$ than the one of the Halton sequence according to \eqref{ineq:hammer}. Hence we may only choose $\alpha = (1-1/N_0)/(1-b)$. In our numerical example this means that we may only jump to $5,346$, which is however an acceptable difference.
\end{remark}
It was conjectured in \cite{NW10} that $n=10d$ points suffice to reach $D_N^*= 0.25$ in dimension $d$. Our result constitutes partial progress towards this conjecture by a simple application of Theorem~\ref{thm:improved:c2}.
\begin{corollary} For every $d \in \mathbb{N}$ there exists a point set $S_1$ with $N=98d$ elements such that $D_N^*(S_1) \leq 0.25$. Moreover, there exists a point set $S_2$ with $N=10d$ elements with $D_N^*(S_2) \leq 0.7789269$.
\end{corollary}
With the value $c=2.4968$ from \cite{GPW21}, we would only obtain a set size $N=100d$ for $S_1$ and $D_N^*(S_2) \leq 0.78956$.

\bibliography{literatur}
\bibliographystyle{abbrv}

\end{document}